\documentclass[hidelinks,onefignum,onetabnum]{siamart251216}

\usepackage{lipsum}
\usepackage{amsfonts}
\usepackage{graphicx}
\usepackage{epstopdf}
\ifpdf
  \DeclareGraphicsExtensions{.eps,.pdf,.png,.jpg}
\else
  \DeclareGraphicsExtensions{.eps}
\fi

\newsiamremark{remark}{Remark}
\newsiamthm{claim}{Claim}

\headers{HDSA with respect to model discrepancy: Sequential OED}{M. Madhavan, J. Hart, and B. van Bloemen Waanders}

\title{Hyper-differential sensitivity analysis with respect to model discrepancy: Sequential optimal experimental design 
}
\author{
Madhusudan Madhavan\thanks{Sandia National Laboratories, mmadhav@sandia.gov} \and 
Joseph Hart\thanks{Sandia National Laboratories, joshart@sandia.gov} \and 
Bart van Bloemen Waanders\thanks{Sandia National Laboratories, bartv@sandia.gov} 
}

\usepackage{amsopn}

\usepackage{lipsum}
\usepackage{epstopdf}
\usepackage{algpseudocode}
\ifpdf
  \DeclareGraphicsExtensions{.eps,.pdf,.png,.jpg}
\else
  \DeclareGraphicsExtensions{.eps}
\fi
\usepackage{float}
\usepackage{booktabs} 
\usepackage{bm}
\usepackage{setspace}
\usepackage{enumitem}
\usepackage{graphicx}
\usepackage[mathscr]{euscript}
\usepackage{cite}
\usepackage{mathtools}
\usepackage{amsopn}
\usepackage{multicol}
\usepackage{xcolor}
\usepackage{array}

\renewcommand{\figurename}{Figure}

\newcommand{\boldheading}[1]{~\\\indent\textit{{#1}.}}

\DeclareMathOperator{\tr}{Tr}

\newcolumntype{P}[1]{>{\centering\arraybackslash}p{#1}}
\newcolumntype{M}[1]{>{\centering\arraybackslash}m{#1}}

\renewcommand{\vec}[1] {\ensuremath{\boldsymbol{#1}}}
\newcommand{\mat}[1]{\vec{#1}}
\newcommand{\mc}[1]{\mathcal{#1}}
\def\R{{\mathbb R}}
\def\U{{\mathcal U}}
\def\Z{{\mathcal Z}}
\def\t{{\vec{\theta}}}
\def\u{{\vec{u}}}

\def\z{{\vec{z}}}
\def\zbar{{\bar{\z}}}
\def\ztilde{{\tilde{\z}}}
\def\x{{\vec{x}}}
\def\e{{\vec{e}}}
\def\v{{\vec{v}}}

\def\y{{\vec{y}}}

\def\b{{\vec{b}}}

\def\d{{\delta}}
\def\A{{\vec{A}}}
\def\B{{\vec{B}}}

\def\M{{\vec{M}}}

\def\G{{\vec{G}}}

\def\L{{\vec{L}}}
\def\C{{\vec{C}}}

\def\H{{\vec{H}}}
\def\P{{\vec{P}}}

\def\V{{\vec{V}}}

\def\I{{\vec{I}}}
\def\g{{\vec{g}}}
\def\W{{\vec{W}}}

\def\F{{\vec{F}}}
\DeclareMathAlphabet\mathbfcal{OMS}{cmsy}{b}{n}

\graphicspath{{./figures/}}
\newcommand{\Exp}{\mathbb{E}}
\newcommand{\diff}[2]{\frac{\partial #1}{\partial #2}}
\newcommand{\difft}[2]{\frac{\partial^2 #1}{\partial #2^2}}
\renewcommand{\tt}{\tilde{\t}}
\newcommand{\St}{\tilde{S}}
\newcommand{\tbar}{\bar{\t}}
\newcommand{\zt}{\tilde{\z}}
\renewcommand{\b}{\vec{\beta}}
\newcommand{\aoed}{\alpha_{k}}

\begin{document}

\maketitle

\begin{abstract}
Large-scale optimization problems are ubiquitous in the physical sciences; yet, high-fidelity models can often be complex and computationally prohibitive for optimization. A practical alternative is to use a low-fidelity model to facilitate optimization. However, the discrepancy between the high- and low-fidelity models can lead to suboptimal solutions. To address this, we build on recent work in Hyper-Differential Sensitivity Analysis \cite{HDSA_MD_3} to leverage limited high-fidelity simulations to update the optimization solution. Our contributions in this article include: (i) incorporating pseudo-time continuation techniques to efficiently compute higher-accuracy optimal solution updates, and (ii) proposing a Bayesian framework for sequential data acquisition that strategically guides high-fidelity evaluations and reduces uncertainty in the model discrepancy estimation. Numerical results demonstrate that our framework delivers significant improvements to optimization solutions with only a few high-fidelity evaluations. 
\end{abstract}



\section{Introduction}
\label{sec:intro}

Optimization of large-scale problems is crucial to numerous domains in science and engineering, where accurate modeling and efficient decision-making are paramount. High-fidelity computational models play a critical role in these efforts, enabling detailed representations of complex systems by capturing intricate physical phenomena and multi-scale interactions. However, their computational demands often render them impractical for direct use in optimization tasks, due to the necessity of repeated evaluations and derivative information. Simpler models are therefore constructed through omission of physical processes, projection into lower dimensions, or data-driven model approximations. We will refer to these as low-fidelity models. An optimal solution obtained using a low-fidelity model can however be grossly suboptimal for the high-fidelity model due to model discrepancy---the difference between the high- and low-fidelity models. 

Hyper-differential sensitivity analysis with respect to model discrepancy (HDSA-MD), introduced in \cite{HDSA_MD_2,HDSA_MD_3}, tackles this challenge by first estimating the model discrepancy using a few high-fidelity evaluations and then approximating the perturbed optimal solution via post-optimality sensitivity analysis. This approach utilizes a Bayesian framework to quantify uncertainty and assess solution quality. However, there are a few remaining challenges present in this framework that need to be addressed, notably, the mitigation of linearization error stemming from post-optimality sensitivities and the efficient acquisition of high-fidelity evaluations. The present work builds on the HDSA-MD framework by addressing these limitations. Specifically, we utilize pseudo-time continuation methods \cite{Continuation1,Continuation2} to improve the quality of the learned optimal solution and develop a novel sequential uncertainty-driven approach for guiding high-fidelity model evaluations.

The HDSA-MD framework is related to work on multi-fidelity modeling and multi-fidelity optimization \cite{ben,multifidelity_quasinewton_bryson,agarwal_trust-region_2013,march_provably_2012,trmm_lewis}. Specifically, this framework shares the goal of using both low- and high-fidelity models to perform optimization and uncertainty quantification (UQ) more efficiently. Unlike traditional multi-fidelity optimization methods, HDSA-MD does not need access to derivatives of the high-fidelity model and allows for offline high-fidelity model evaluations. Our approach is also related to work on Bayesian calibration of model discrepancy, as pioneered by Kennedy and O'Hagan in the seminal work \cite{Ohagan2001} and subsequently advanced by recent works \cite{ohagan2014,Higdon_2008,Maupin,Arendt_2012,Ling_2014}. Although conceptually similar, HDSA-MD focuses on deterministic optimization problems and considers discrepancy as the difference between high- and low-fidelity models, rather than the traditional calibration context of model discrepancy studies. Moreover, we assume high-fidelity models can be evaluated at different user-specified optimization variable inputs.

At its core, HDSA-MD utilizes post-optimality sensitivity analysis \cite{shapiro_SIAM_review,HDSA} to approximate the mapping from the model discrepancy to the optimization solution. While previous works \cite{HDSA_MD_2,HDSA_MD_3} considered a linear approximation of this mapping, linearization may not be sufficient when the model discrepancy is large or when the mapping is highly nonlinear. To address limitations arising from linearization, we consider continuation and path-following methods \cite{allgower_georg,allgower_georg_siam,param_opt_book,Liu_2025,Si_2024}. Notably, the recent works \cite{Continuation1,Continuation2} introduce a computationally efficient pseudo-time continuation approach for parameterized PDE-constrained optimization problems. In the present work, we utilize this continuation approach to efficiently compute the model discrepancy to optimization solution mapping. 

Optimal experimental design (OED) has been well-studied in literature, with a rich body of theoretical and methodological work spanning several decades \cite{atkinson_optimum_1992,pukelsheim_optimal_2006}. In particular, the Bayesian approach to OED \cite{alexanderian_optimal_2021,huan_optimal_2024} specifies experimental settings in order to minimize uncertainty (or maximize expected information). We draw on this approach to guide data collection for HDSA-MD. Specifically, we adopt a goal-oriented OED approach \cite{Attia_2018,wu_offline-online_2023} to strategically guide the high-fidelity data acquisition. Note, however, that experimental data in the present context refers to high-fidelity evaluations and not necessarily physical measurements. We utilize a sequential approach to optimal data collection, drawing parallels with techniques for data acquisition in Bayesian Optimization (BO)\cite{wang_recent_2023,shahriari_taking_2016}. However, our approach is set in the context of a high-dimensional optimization problem where there is access to a low-fidelity model but limited access to the high-fidelity model, rendering it infeasible for traditional BO methods. Our approach also shares some similarities with literature on sequential optimal experimental design \cite{gautier_adaptive_2000,surer_sequential_2024,huan_numerical_2015,go_sequential_2025} yet differs considerably as our strategy involves acquiring data to update a low-fidelity optimization solution. 

The key contributions of this article are:
\begin{itemize}
  \item The application of pseudo-time continuation techniques to enable efficient and accurate post-optimal solution updates.
  \item The development of a sequential uncertainty-driven approach for optimal data acquisition to enable computationally efficient and effective model discrepancy calibration with limited evaluations.
  \item Illustrative numerical experiments that showcase the potential of the proposed approach.
\end{itemize}

The article is organized as follows. In \Cref{sec:background}, we review the HDSA-MD framework and highlight its limitations that we improve upon in this article. Our contributions are given in the subsequent three sections. First, the pseudo-time continuation method for optimal solution updates is presented in \Cref{sec:continuation}. Second, the proposed framework for optimal design of experiments in the context of model-discrepancy calibration is developed in \Cref{sec:OED}. Third, numerical tests are performed in \Cref{sec:results}, demonstrating the effectiveness of the approaches in the context of diffusion-reaction and fluid flow model problems. \Cref{sec:conclusion} presents concluding remarks and highlights potential directions for future work.

\section{Background}
\label{sec:background}
Consider an optimization problem of the form, 
\begin{equation}
  \min_{z \in \Z} J(S(z), z),
  \label{eq:hifi_optim_prob}
\end{equation} 
where $z$ belongs to a Hilbert space $\Z$. The map $S: \Z \to \U$ is the solution operator for a differential equation with state variable $u$ in an infinite-dimensional Hilbert space $\U$, and the objective function is $J: \U \times \Z \to \R$, which we seek to minimize. In many situations, solving \cref{eq:hifi_optim_prob} directly may be infeasible due to the complexity of the high-fidelity solution operator $S$. Therefore, we instead solve, 
\begin{equation}
  \min_{z \in \Z} J(\St(z), z),
  \label{eq:lofi_optim_prob}
\end{equation} 
where $\St: \Z \to \U$ is the solution operator corresponding to a low-fidelity model. The goal of HDSA-MD is to use a few high-fidelity model evaluations to update the minimizer of \cref{eq:lofi_optim_prob}. 

\boldheading{Notation}
The state space $\U$ is infinite-dimensional, and the optimization variable space $\Z$ may be finite- or infinite-dimensional. To discretize, we define the bases \(\{\eta_i\}_{i=1}^{n_u}\) spanning a subspace \(\U_h \subset \U\) and \(\{\varphi_j\}_{j=1}^{n_z}\) spanning a subspace \(Z_h \subset Z\). Note that if $\Z$ is finite-dimensional, we simply let $\Z_h = \Z$. This allows for the finite-dimensional coordinate representation of $u \in \U$ as $\u \in \R^{n_u}$ and $z \in \Z$ as $\z \in \R^{n_z}$, where the discretized dimensions $n_u$ and $n_z$ are typically large. We define the mass matrices, $\M_{\u} \in \R^{n_u \times n_u}$ and $\M_{\z} \in \R^{n_z \times n_z}$, whose $(i, j)$th elements are defined by the Hilbert space inner products, $\langle\eta_i, \eta_j\rangle$ and $\langle\varphi_i, \varphi_j\rangle$, respectively. We henceforth adopt this discretized finite-dimensional setting, with the objective function \(J : \mathbb{R}^{n_u} \times \mathbb{R}^{n_z} \to \mathbb{R}\) and solution operators \(S : \mathbb{R}^{n_z} \to \mathbb{R}^{n_u}\) and \(\tilde{S} : \mathbb{R}^{n_z} \to \mathbb{R}^{n_u}\). We use \(\nabla_\u\) and \(\nabla_\z\) to denote the gradient of a scalar-valued function (or Jacobian of a vector-valued function) with respect to \(\u\) and \(\z\), and \(\nabla_{\u,\z}\) and \(\nabla_{\z,\z}\) to denote second-order derivative matrices with respect to \((\u,\z)\) and \((\z,\z)\), respectively. 

In the following sections, we provide a brief overview of the HDSA-MD approach. Namely, the HDSA-MD framework is composed of two fundamental steps: (i) calibrating a representation of the discrepancy between the models, $S(\z) - \St(\z)$, and (ii) updating the optimization problem solution. 
To quantify the sensitivity of the optimal solution to model discrepancy, we employ a post-optimality sensitivity approach, as outlined in \Cref{ssec:optim_form}. This analysis allows for updating the optimizer upon calibration. This is followed by an overview of the model calibration process in \Cref{ssec:bayes_ip}. We finally discuss the limitations of the current HDSA-MD approach in \Cref{ssec:limitations}, which sets the stage for the developments in this article.

\subsection{Post-optimality sensitivity analysis}
\label{ssec:optim_form}
To incorporate model discrepancy into optimization, consider the parameterized optimization problem \cite{HDSA_MD_3},
\begin{align}
\label{eq:minJ_z_theta}
 \min_{\z \in \R^{n_z}} \hspace{1 mm} \mc{J}(\z,\t):=J(\tilde{S}(\z)+\d(\z,\t),\z) ,
\end{align}
where $\d:\R^{n_z} \times \R^{n_\theta} \to \R^{n_u}$ is a mapping that approximates the discrepancy between the high-fidelity and low-fidelity solution operators. This mapping, which we will define in~\Cref{ssec:bayes_ip}, is parameterized by a high-dimensional vector $\t \in \R^{n_\theta}$ and calibrated so that $\d(\z, \t) \approx S(\z) - \St(\z)$. Assume that $\d$ is parameterized such that $\d(\z, \vec{0}) = \vec{0}$ for all $\z$, that is, $\mc{J}(\z,\vec{0})$ is the low-fidelity optimization objective. To facilitate subsequent derivations, assume that $\mc{J}(\z,\t)$ is twice continuously differentiable with respect to $(\z,\t)$. Let $\zt \in \R^{n_z}$ denote a local minimizer of the low-fidelity problem, satisfying the first-order and second-order optimality conditions:
\begin{equation}
\nabla_\z \mathcal{J}(\zt,\vec{0}) = \vec{0},
\label{eq:optim_condition}
\end{equation}
and $\nabla_{\z, \z} \mathcal{J}(\zt,\vec{0})$ is positive-definite, respectively.
Applying the Implicit Function Theorem to \cref{eq:optim_condition} yields the existence of an operator $\F: \mc{N}(\vec{0}) \to \R^{n_z}$, defined on a neighborhood $\mc{N}(\vec{0})$ of $\vec{0} \in \R^{n_\theta}$, such that
\begin{align*}
\nabla_\z \mc{J}(\F(\t),\t) = \vec{0}.
\end{align*}
If the Hessian $\nabla_{\z, \z} \mc{J}(\F(\t),\t)$ is positive-definite for all $\t \in \mc{N}(\vec{0})$, then the operator $\F$ can be interpreted as a local mapping from a given calibration estimate $\t \in \mc{N}(\vec{0})$ to the solution of \cref{eq:minJ_z_theta}. Moreover, the Jacobian of $\F$ with respect to $\t$, evaluated at $\t=\vec{0}$, is given  by
\begin{eqnarray}
\label{eq:dis_sen_op}
\nabla_\t \F(\vec{0}) = - \H^{-1} \B \in \R^{n_z \times n_\theta},
\end{eqnarray}
where $\H =\nabla_{\z,\z} \mc{J}(\tilde{\z},\vec{0}) \in \R^{n_z \times n_z}$ is the Hessian of $\mc{J}$ and $\B = \nabla_{\z,\t} \mc{J}(\tilde{\z},\vec{0}) \in \R^{n_z \times n_\theta}$ is the Jacobian of $\nabla_\z \mc{J}$ with respect to $\t$, both evaluated at the low-fidelity optimizer \cite{HDSA, Griesse2006}. 
The post-optimality sensitivity operator, $\nabla_\t \F(\vec{0})$, represents the local sensitivity of the optimal solution with respect to the model discrepancy. 

\boldheading{Hessian Projection}
Our goal is to utilize the post-optimality sensitivity operator, $\nabla_\t F(\vec{0})=-\H^{-1} \B$, to update the optimal solution $\zt$. However, in many applications, the objective function only exhibits sensitivity in a low-dimensional subspace. In particular, observe that near the local optimizer $\zt$, the objective function has the greatest sensitivity in directions that align with the largest eigenvectors of the Hessian. Conversely, perturbations that occur in directions of the small eigenvalues of $\H$ have little meaningful impact on the optimization objective. However, these directions correspond to the largest eigenvalues of $\H^{-1}$, and hence direction of large post-optimality sensitivity. Thus, perturbations of the parameter $\t$ in directions in which $\B \t$ aligns with the small eigenvalues of $\H$ will yield large optimal solution updates but small changes in the objective function value. Furthermore, since $\t$ will be calibrated using limited data, we may have noisy estimates of $\t$ such that $\B \t$ aligns with the small eigenvalues, thus leading to high sensitivity to noise. 

To avoid this, we implicitly regularize the optimal solution update by projecting it onto a subspace of high objective function sensitivity. This also has the benefit of reducing computational complexity by reducing dimensionality of the optimization space. To define the projector, we compute the eigenvectors of $\H$ in the $\W_\z$ weighted inner product:
\begin{align}
\label{eqn:hess_gen_eig}
\H \vec{v}_j = \rho_j \W_\z \vec{v}_j, \qquad j=1,2,\dots,r,
\end{align}
where $r$ is the truncation rank (the eigenpairs are ordered with $\rho_1 \ge \rho_2 \ge \cdots$). The weighting matrix $\W_\z \in \R^{n_z \times n_z}$, which will be specified in~\Cref{ssec:bayes_ip}, imposes prior information arising from our Bayesian discrepancy calibration formulation. In practice, the eigenpairs are estimated with iterative or randomized methods \cite{Grimes1994,saibaba_gsvd}. 

Let $\vec{V} = \begin{pmatrix} \vec{v}_1 & \vec{v}_2 & \cdots \vec{v}_r \end{pmatrix} \in \R^{n_z \times r}$ be an orthogonal matrix in the $\W_\z$ weighted inner product and define the projector as $\P = \vec{V} \vec{V}^\top \W_\z$. We may thus approximate the optimal solution update using a linearization of the form, 
\begin{equation}
  \label{eq:lin_sol}
  \F(\t) \approx \F(\vec{0}) + \P \ \nabla_\t \F(\vec{0}) \t = \zt - \P \H^{-1} \B\t.
\end{equation}

\subsection{Bayesian model discrepancy calibration}
\label{ssec:bayes_ip}
\boldheading{Representation of model discrepancy}
We seek to calibrate $\t$ 
so that $\d(\z, \t) \approx S(\z) - \St(\z)$. Since the high-fidelity model can only be evaluated a small number of times,
we follow \cite{HDSA_MD_2} and consider a model discrepancy representation that is an affine function of $\z$:
\begin{equation}
  \label{eq:slope_char_discrep}
  \d(\z,\t) = \mc{I}(\t) + \mc{K}(\t) \M_{\z} \z,
\end{equation}
where $\mc{I}(\t) \in \R^{n_u}$ extracts the first $n_u$ elements of $\t$, and $\mc{K}(\t) \in \R^{n_u \times n_z}$ is a row-wise reshaping of the remaining elements of $\t$; the mass matrix $\M_\z$ ensures that the representation is discretization-invariant. The representation in \Cref{eq:slope_char_discrep} can be succinctly expressed in terms of the Kronecker product with the identity matrix, $\I_{n_u} \in \R^{n_u \times n_u}$, as
\begin{equation}
\d(\z,\t) =
\left( \begin{array}{cc}
\I_{n_u} & \I_{n_u} \otimes \z^\top \M_{\z}
\end{array} \right) \t ,
\label{eq:delta_kron}
\end{equation}
where the parameter dimension is $n_\theta=n_u(n_z+1)$. This Kronecker product representation is key to enabling efficient computations with the discrepancy parameterization, reducing computations in $\R^{n_\theta}$ to those in $\R^{n_u}$ and $\R^{n_z}$. While the affine model \cref{eq:slope_char_discrep} provides a simple and useful characterization of local discrepancy sufficient for most data-sparse applications, we mention that other representations that are linear in $\t$ can also be considered in place of \cref{eq:delta_kron} with appropriate modifications (see \Cref{sec:conclusion}).

The goal of calibration is to utilize discrepancy evaluations to estimate $\t$. However, limitations on the number of high-fidelity evaluations leads to a highly underdetermined calibration problem where there are infinitely $\t$'s such that $\d(\z,\t)$ exactly fits the model discrepancy data. Therefore, we turn to a Bayesian framework, in which we utilize a prior distribution for the unknown parameter $\t$, to obtain an informative \textit{posterior} distribution for $\t$. The prior guides the calibration of model discrepancy in directions that are not well-informed by the data and hence is crucial in data-limited settings.

\boldheading{Defining the prior distribution}
To build on existing theory and enable closed-form expressions for the posterior, we specify a mean-zero Gaussian prior for $\t$ with covariance $\W_\t^{-1} \in \R^{n_\theta \times n_\theta}$. Specifically, we aim to construct the precision matrix $\W_\t$ in a manner that reflects prior belief on the discrepancy operator. To this end, let $\W_\u^{-1} \in \mathbb{R}^{n_u \times n_u}$ be a state space covariance matrix whose inverse induces a norm suitable for measuring elements in the discrepancy range, i.e., the mean-zero Gaussian measure with covariance $\W_\u^{-1}$ models the discrepancy range space. Similarly, let $\W_\z^{-1} \in \mathbb{R}^{n_z \times n_z}$ be an optimization variable space covariance matrix such that a mean-zero Gaussian distribution with covariance $\W_\z^{-1}$ assigns high probability to physically plausible perturbations of the optimization solution. These matrices may be determined by physical properties of the underlying models. In our numerical experiments (\Cref{sec:results}), we define these covariances as inverses of a Laplacian-like differential operator, as is common in Bayesian inverse problem literature \cite{BuiThanh2013,Huang2022,Stuart2010}. We refer to \cite{HDSA_MD_HP} for additional guidance on prior operators.

Since the discrepancy representation $\d(\z, \t)$ is an affine function of $\z$, it can be characterized by a point evaluation and its Jacobian. Since the post-optimality sensitivity analysis is centered at $\zt$, we define the precision matrix $\W_\t$ such that it regularizes the discrepancy representation as,
\[\t^\top \W_\t \t = ||\d( \zt, \t)||^2_{\W_\u} + ||\nabla_\z \d(\cdot,\t) ||^2_{\mathcal L(\M_{\z} \W_{\z}^{-1} \M_{\z},\W_\u)}\]
where the norm on $\nabla_\z \d(\cdot,\t)$ is a Schatten 2-norm with the domain and range weighted by $\M_{\z} \W_{\z}^{-1} \M_{\z}$ and $\W_\u$, respectively. The inner product choice is chosen to facilitate mesh refinement properties and adequate smoothness in prior discrepancy \cite{HDSA_MD_HP}. This choice results in the precision matrix
\begin{eqnarray}
\label{eq:W_theta}
\W_\t =  \left( \begin{array}{cc}
\W_\u & \W_\u \otimes \ztilde^\top \M_{\z} \\
\W_\u \otimes \M_{\z} \ztilde & \W_\u \otimes \W_\z+ \M_\z \zt \zt^\top \M_\z 
\end{array} \right)  \in \R^{n_\theta \times n_\theta}.
\end{eqnarray} 

\boldheading{Fitting discrepancy data}
We wish to estimate $\t$ using $N$ discrepancy evaluations. Accordingly, let $\{\z_\ell \}_{\ell=1}^{N} \subset \R^{n_z}$ be a set of vectors, which, to facilitate subsequent derivations, we assume to be linearly independent. We acquire discrepancy data $\vec{d}_\ell := S(\z_{\ell}) - \tilde{S}(\z_{\ell})$ at each of the data points $\z_{\ell}$ and aim to fit $\t$ such that $\d(\z_{\ell}, \t) \approx \vec{d}_{\ell}$ for $\ell = 1, 2, \dots, N$. We formulate this more precisely as follows. Let
\begin{eqnarray}
\label{eq:Aell}
\A_\ell = \begin{pmatrix} \I_{n_u} & \I_{n_u} \otimes \z_\ell^\top \M_{\z} \end{pmatrix} \in \R^{n_u \times n_\theta}, \qquad \ell=1,2,\dots,N,
\end{eqnarray}
so that $\d(\z_\ell,\t)=\A_\ell \t$, and consider the concatenation of these matrices
\begin{eqnarray*}
\A = \left(
\begin{array}{c}
\A_1 \\
\A_2 \\
\vdots \\
\A_N
\end{array}
\right) \in \R^{n_uN \times n_\theta}.
\end{eqnarray*}
This way, $\A \t \in \R^{n_uN}$ corresponds to the evaluation of $\d(\z,\t)$ for all of the input data. In an analogous fashion, let $\vec{d} \in \R^{n_uN}$ be defined by stacking $\vec{d}_\ell$, $\ell=1,2,\dots,N$, into a vector. We ultimately seek $\t$ such that $\A \t \approx \vec{d}$. 

\boldheading{Defining the noise model}
In Bayesian inference, data often arises from physical observations or measurements, which are inherently prone to measurement error or limited precision. In such cases, observation noise is modeled using physical insights or heuristics. However, in the present context, data pertains to the simulation of the model discrepancy at different inputs. Although this type of data may not contain experimental error, there is error due to fitting an affine operator $\d(\z, \t)$ to approximate a potentially nonlinear operator, $S(\z) - \St(\z)$. Therefore, we specify the noise model in a way that reflects our belief about how well the affine map should fit the data. To enable a closed-form expression for the posterior, we consider an additive mean-zero Gaussian noise model. Based on our interpretation of noise as linearization error, we define the noise covariance matrix as $\W_{\vec{d}}^{-1} = \alpha_{\vec{d}} (\I_N \otimes \M_\u^{-1})$, so that the noise precision matrix defines a norm scaled by $\alpha_{\vec{d}}^{-1} > 0$ to control the magnitude of the misfit error. Note that this formulation considers the linearization error at different evaluations to be independent. 

Thanks to the linear dependence of $\d(\z, \t)$ on $\t$ and Gaussian characterization of noise and prior models, the posterior distribution of $\t$ is Gaussian with its mean and covariance given by
\begin{equation}
\label{eq:post_mean_product}
\tbar =  \mat{\Sigma}_\t \A^\top \W_{\vec{d}}\ \vec{d}, \quad \text{and} \quad \mat{\Sigma}_\t = \left( \W_\t + \A^\top \W_{\vec{d}} \A \right)^{-1},
\end{equation}
respectively. The maximum-a-posteriori point (MAP) estimate $\tbar$ provides a best-guess estimate of $\t$ based on prior belief and discrepancy data, while $\mat{\Sigma}_\t$ enables the quantification of uncertainty in $\t$. The Kronecker structure of the prior operators and $\A$ allows for the efficient representation of the posterior covariance and enables fast sampling from the posterior distribution; see Appendix~\ref{ssec:app_post_fact} and \cite{HDSA_MD_3} for details.

\subsection{Limitations with existing methods} \label{ssec:limitations}
The HDSA-MD framework outlined in \Cref{ssec:optim_form,ssec:bayes_ip} enables efficient updates of optimal solutions using limited evaluations of the high-fidelity model. 
Nevertheless, there are key limitations with the existing framework that restrict its potential. 

One of the major shortcomings is that the mapping used to update the optimal solution from the model discrepancy parameterization is approximated by a linear operator, namely the post-optimality sensitivity operator at the low-fidelity optimal solution. However, the optimal solution depends nonlinearly on $\t$. Therefore, the linear approximation \cref{eq:lin_sol} may be inaccurate for large discrepancies. Second, the effectiveness of the model calibration and optimal solution update hinges on the availability of high-quality data near the high-fidelity optimizer. However, the high-fidelity discrepancy data available is often very scarce (e.g., $N < 5$). Therefore, it is crucial that the data is collected strategically, in a way that is tailored to the HDSA-MD framework, in order to gain the most from the limited evaluations. The goal of this paper is to address these two limitations. \Cref{sec:continuation} aims to mitigate the linearization error in a post-optimality sensitivity update through pseudo-time continuation, and \Cref{sec:OED} introduces a sequential OED framework for the optimal acquisition of discrepancy data.

\section{Psuedo-time continuation}
\label{sec:continuation}
In this section, we mitigate the linearization error occurring in the post-optimality update by utilizing a pseudo-time continuation approach \cite{Continuation1,Continuation2}. Recall the projected optimal solution update \cref{eq:lin_sol} and observe that it may be viewed as an approximation to the solution of a projected optimization problem. In particular, since we are primarily interested in directions that the objective function $J$ is sensitive to, we focus on obtaining a solution to~\cref{eq:minJ_z_theta} that lies in directions defined by the range of the projector $\P$. This transforms the problem of optimizing \cref{eq:minJ_z_theta} over $\z \in \R^{n_z}$ to 
\begin{equation}
  \label{eq:minJ_b_theta}
  \min_{\b \in \R^r} \mc{J}(\z_{\b}, \t) := J(\tilde{S}(\z_{\b}) + \delta(\z_{\b}, \t), \z_{\b}), \qquad \z_{\b} := \zt + \V\b.
\end{equation}
Performing post-optimality sensitivity analysis to obtain the local sensitivity of the optimal $\b$ to $\t$ in \cref{eq:minJ_b_theta} yields the same update as in \cref{eq:lin_sol}. However, by utilizing pseudo-time continuation on \cref{eq:minJ_b_theta}, we mitigate the effect of linearization error and hence improve the optimization solution update.

Recall the parameter estimate $\tbar$ \cref{eq:post_mean_product} obtained through the discrepancy calibration and form a line segment from $\vec{0}$ to $\tbar$,
\[\t(t) = t\tbar, \quad t \in [0, 1].\]
Let $\b^*(t)$ denote the minimizer of \cref{eq:minJ_b_theta} for $\t = \t(t)$. Note that for a non-convex optimization problem, the minimizer may not be unique. However, if the Hessian is positive definite for all $t \in [0,1]$, $\b^*(t)$ is uniquely defined by restricting to a neighborhood of the minimizers from times $0 \le s<t$. For each value of $t$, using the Implicit Function Theorem and an analogous argument to that in \Cref{ssec:optim_form}, we obtain the post-optimality sensitivity of $\b^*$ with respect to $\t$, centered at $\t(t)$.
Using the chain rule, and through abuse of notation writing $\mathcal{J}$ as a function of $\b$ rather than $\z$, we can express 
\[\frac{d \b^*}{dt} = \frac{\partial \b^*}{\partial \t} \cdot \frac{d \t}{dt} = -\nabla_{\b,\b} \mathcal{J}(\b^*(t), \t(t))^{-1}\ \nabla_{\b,\t} \mathcal{J}(\b^*(t), \t(t)) \tbar.\]
Therefore, we may obtain $\b^*(t)$ for $t \in (0, 1]$ by solving the ordinary differential equation (ODE), 
\begin{align}
  \label{eq:cont_ivp}
  \frac{d \b^*}{dt} &= -\nabla_{\b,\b} \mathcal{J}(\b^*(t), \t(t))^{-1}\ \nabla_{\b,\t} \mathcal{J}(\b^*(t), \t(t)) \ \tbar, \\
  \b^*(0) &= \vec{0}. \nonumber
\end{align}
Our goal is to obtain the optimal solution at $\tbar$, that is, obtain $\b^*(1)$. Equation \cref{eq:lin_sol} corresponds to an approximation of the solution to \cref{eq:cont_ivp} via a single forward Euler time step with step size $\Delta t = 1$. To improve the solution, we deploy a time integration scheme with $\Delta t < 1$ to solve the ODE~\cref{eq:cont_ivp}. However, we observe that \cref{eq:cont_ivp} possesses additional structure due to the first-order optimality condition, $\nabla_{\b} \mathcal{J}(\b^*(t), \t(t))=\vec{0}$ for all $t$. This motivates the use of specialized predictor-corrector time integration schemes. We leverage the Euler-Newton method that predicts with a forward Euler time step and corrects with a Newton optimization step that seeks to enforce $\nabla_{\b} \mathcal{J}(\b^*(t), \t(t))=\vec{0}$. 

Let $N_c$ denote the number of continuation steps (or time steps from the ODE perspective). For $n = 0, 1, \dots, N_c$, define $t_n=n {\Delta t}$, $\t_n = t_n \tbar$, let $\b_n$ denote the numerical approximation of $\b^*(t_n)$, and $\b_0=\vec{0}$. The optimal solution $\b^*$ at $t_{n+1}$ is approximated by $\b_{n+1}$, via the predictor-corrector scheme\footnote{If $N_c$ is too small, then the predictor-corrector scheme may leave the Newton basin of attraction and the iteration will not converge. In such cases, increasing $N_c$ sufficiently will remedy the issue.}
\begin{equation}
  \label{eq:ctn2}
    \begin{cases}
        \b_{n+1}^{\textrm{pred}} = \b_n -  {\Delta t} \left(\nabla_{\b,\b} \mathcal{J}(\b_{n}, \t_{n})\right)^{-1} \nabla_{\b,\t}\mc{J}(\b_{n}, \t_{n}) \tbar , \\[1ex]
        \b_{n+1} = \b_{n+1}^{\textrm{pred}} - \left(\nabla_{\b,\b} \mathcal{J}(\b_{n+1}^{\textrm{pred}}, \t_{n+1})\right)^{-1} \nabla_{\b}\mc{J}(\b_{n+1}^{\textrm{pred}}, \t_{n+1}).  \\[2ex]
      \end{cases}
\end{equation}
Iterating \cref{eq:ctn2} through $n = 0, 1, \dots, N_c - 1$, we obtain $\bar{\b} := \b_{N_c}$ and the corresponding optimal solution update, $\zbar = \zt + \mat{V}\bar{\b}$.
While the original linearization-based approach is adequate when the discrepancy lies within a certain local neighborhood, it struggles to perform well when further away. The continuation algorithm addresses this by instead sequentially performing linearization with a smaller step, $\Delta t \bar{\t}$, to obtain the next iterate. Moreover, the corrector at each iterate enhances the quality of solutions and reduces the accumulation of error over iterates. In addition to accurately estimating the solution of~\cref{eq:minJ_b_theta}, the continuation approach~\cref{eq:ctn2} enjoys additional computational benefits. Although the sequence of iterations~\cref{eq:ctn2} has a similar form to Newton steps in an optimization algorithm, they can be executed more efficiently. If the solution of the ODE system~\cref{eq:cont_ivp} also satisfies the second-order optimality condition, which is frequently true in practice, then specialized Hessian preconditioners can be leveraged to accelerate the inverse Hessian computations \cite{Continuation2}. We also point out that the continuation techniques presented in this section can be modified to obtain (pushforward) posterior samples of the optimal solution update, which can be useful for uncertainty quantification. Specifically, using the procedure above with $\overline{\t}$ replaced by a posterior sample $\t^\dagger$ will render an optimal solution posterior sample.

\section{Optimal Experimental Design}
\label{sec:OED}
Given that evaluations of the high-fidelity model are limited, it is crucial that the high-fidelity evaluations are run with inputs $\{ \z_1,\z_2,\dots,\z_N \}$ that provide the most valuable information to enhance the optimization problem solution update. We focus on (i) sequential approaches---using the inference from the discrepancy data evaluations at $\{\z_1,\z_2,\dots,\z_k\}$ to determine a single input $\z_{k+1}$, and (ii) batch sequential approaches---using the inference from the discrepancy data evaluations at $\{ \z_1,\z_2,\dots,\z_k \}$ to determine a batch of inputs $\{ \z_{k+1},\z_{k+2},\dots,\z_{k+p} \}$. Batch sequential approaches are sometimes favored if multiple high-fidelity simulations can be run in parallel. 

The simplest approach, which we refer to as \textit{optimal solution update tracing}, is to sequentially collect data at the optimal solution update itself. Specifically, initializing $\z_1=\zt$, we select $\z_{k+1}$ as the predicted optimization solution attained by HDSA-MD using the data $\{\z_1,\z_2,\dots,\z_k\}$, $k=1,2,\dots$. However, there are a few potential downsides to this approach. One key issue is that such an approach does not consider the uncertainty present in calibration. In other words, there could be relatively low uncertainty in the discrepancy at $\z_{k+1}$, but high uncertainty in whether $\z_{k+1}$ approximates the optimization solution. In such cases, collecting data at the optimal solution update could be rather uninformative. 

As a simple illustrative example, suppose that the discrepancy operator, which in general satisfies $S(\z)-\tilde{S}(\z) \ne \vec{0}$, has a root at $\zt$, i.e., $S(\zt)-\tilde{S}(\zt)=\vec{d}_1=\vec{0}$. Then, the MAP point will default to the zero prior mean ($\tbar = \vec{0}$), and thus the optimal solution update remains at $\zbar_1 = \zt$. If the optimal solution update tracing approach is used for data collection, the algorithm will select $\z_{k+1}=\zt$ for all $k$, leading to no meaningful improvement. Although this particular scenario may not occur, similar situations (e.g., cases where discrepancy magnitude varies significantly throughout the domain) may arise in practice, causing stagnation in performance due to the lack of variation in data. Such examples point to the need for data collection approaches that balance exploration---collecting data in new locations---and exploitation---utilizing existing knowledge to constrain the exploration. The optimal solution update tracing relies too heavily on exploitation without having an option to increase exploration. 

In this article, we utilize an uncertainty-reduction approach to select the next model inputs for discrepancy evaluation. We sequentially acquire data, perform model discrepancy calibration, update the optimal solution, and use this information to determine the next model inputs for discrepancy evaluation. However, our approach differs from the optimal solution update tracing approach in that we leverage a Bayesian framework to trade off exploration and exploitation, in addition to enabling batch sequential designs.

\subsection{Design Setup}
\label{ssec:design_setup}
Assume that the first input is chosen at the low-fidelity optimization solution $\zt$. Although this assumption is not essential, taking $\z_1=\zt$ has both practical and theoretical benefits. Our approach to select the subsequent inputs is formulated inductively as follows. 
Suppose we have collected the first $k$ data points $\{\z_\ell, \vec{d}_\ell \}_{\ell=1}^k$, performed model calibration to obtain the posterior distribution, $\t \sim \mathcal{N}\!\left( \tbar_k, \mat{\Sigma}_{\t,k} \right)$, and computed the optimal solution update, $\zbar_k$. Our goal is to collect the next $p$ data points such that uncertainty in $\delta(\z, \t)$ near $\zbar_k$ is reduced. This can be formulated as an optimization problem of the form 
\begin{equation}
	\min_{\z_{k+1},\z_{k+2}, \dots, \z_{k+p}} \Psi_k(\z_1, \z_2, \dots, \z_{k+p}), 
	\label{eq:OED_crit_optim}
\end{equation}
where $\Psi_k$ is a design criterion measuring uncertainty given high-fidelity model evaluations at the inputs $\z_1, \z_2, \dots, \z_{k+p}$. As a first step, we constrain the data points $\{ \z_{1},\z_{2},\dots,\z_{k+p} \}$ to the range of the projector $\P$ since all updates are constrained to this subspace. Consider data points
\begin{equation}
\z_\ell = \ztilde + \sum\limits_{j=1}^r \eta_j^{\ell} \v_j \qquad \ell=1,2,\dots,k+p.
\label{eq:z_ell_proj}
\end{equation}
Since we only need to optimize over designs $\{\z_{k+1},\z_{k+2},\dots,\z_{k+p} \}$, 
we have a total of $pr$ degrees of freedom to determine a design. Therefore, it suffices to optimize \cref{eq:OED_crit_optim} over the design parameters $\left\{ \eta_j^{\ell} \right\}_{j=1}^r$ where $\ell = k+1, k+2, \dots, k+p$. 

\subsection{OED criterion}
\label{ssec:oed_crit}
We aim to formulate a criterion for optimal experimental design that is tailored for the model discrepancy calibration problem at hand. The Bayesian framework allows us to quantify posterior uncertainty in $\t$. Due to the linear nature of the inverse problem and Gaussian prior and noise assumptions, the covariance of $\t$ depends only on $\{\z_{\ell}\}_{\ell=1}^{k+p}$ (where discrepancy data is collected) and is independent of the discrepancy data $\{ \vec{d}_\ell \}_{\ell=1}^{k+p}$. Therefore, we may strategically acquire data such that uncertainty is minimized in quantities of interest. In particular, we are most interested in reducing uncertainty of the discrepancy near the best-guess solution $\zbar_k$. To this end, we utilize a type of goal-oriented A-optimal criterion tailored to quantify uncertainty near the optimal solution update. This yields a useful statistical interpretation \cite{alexanderian_bayesian_2016,chaloner_bayesian_1995} and facilitates efficient closed-form expressions of the OED criterion. We point out that other classical criteria (e.g., D-optimality) may also be analogously adapted in this context. We define the criterion $\Psi_{k}$ as follows.

\begin{definition}
For a fixed $\z \in \R^{n_z}$, let $\vec{\Sigma}_\d (\z) \in \R^{n_u \times n_u}$ denote the covariance matrix of $\d(\z, \t)$, the pushforward of the posterior distribution of $\t$ through the discrepancy representation $\d(\z,\cdot)$. Define the criterion,
  \begin{equation}
    \Psi_{k}(\z_1, \z_2, \dots, \z_{k+p}) = \mathbb{E}_{\z \sim \mu_{k}}\!\left\{ \tr(\vec{\Sigma}_\d (\z) \M_\u) \right\},
    \label{eq:crit}
  \end{equation}
  where $\mu_{k}$ denotes a Gaussian measure $\mathcal{N}(\zbar_k, \aoed \W_{\z}^{-1})$ with $\aoed > 0$, and $\tr(\cdot)$ denotes the standard Euclidean trace.
  \label{def:criterion}
\end{definition}

The criterion \cref{eq:crit} quantifies posterior uncertainty in the estimated discrepancy in a local neighborhood of the optimal solution update, $\zbar_k$. Therefore, acquiring data that minimizes $\Psi_{k}$ enables targeted and optimal uncertainty reduction. We also highlight that the trace in \cref{eq:crit} is weighted with $\M_\u$ to ensure adequate mesh scaling. The criterion can be also expressed in an information-theoretic sense as the expected Bayes risk, as seen in \Cref{thm:BayesRisk}. 

\begin{theorem}
  Let $k$ and $p$ be positive integers. 
  The data acquisition criterion from Definition~\ref{def:criterion} satisfies,
  \[\Psi_{k} = \Exp_{\z \sim \mu_k}\left\{ \Exp_{\mu_{\mathrm{pr}}}\left\{ \Exp_{\vec{d} | \t}\left\{ \|\d(\z, \tbar_{k+p}) - \d(\z, \t)\|_{\mat{M}_u}^2 \right\} \right\} \right\},\]
  where $\mu_{\mathrm{pr}} = \mathcal{N}\!\left( \vec{0}, \mat{W}_{\t}^{-1} \right)$ denotes the prior distribution of $\t$, $\tbar_{k+p}$ denotes the posterior mean given data inputs $\{\z_1, \z_2, \dots, \z_{k+p}\}$ and corresponding discrepancy data $\vec{d}$.
  \label{thm:BayesRisk}
\end{theorem}

\begin{proof}
  The proof of this result mostly follows from \cite[Proposition 3.1]{Attia_2018}. Namely, we express $\d(\z, \t) = \mat{A}(\z)\t$, where $\mat{A}(\z)$ is a matrix with full row-rank, and express
  \[ \tr(\vec{\Sigma}_\d (\z)\M_{\u}) = \Exp_{\mu_{\mathrm{pr}}}\left\{ \Exp_{\vec{d} | \t}\left\{ \|\d(\z, \tbar_{k+p}) - \d(\z, \t)\|_{\mat{M}_u}^2 \right\} \right\}.\]
  Taking expectation over $\z \sim \mu_k$ yields the desired result.
\end{proof}

\Cref{thm:BayesRisk} indicates that the OED criterion is expressible as the expected squared deviation of the discrepancy calibration in a neighborhood of the \textit{best-guess} optimal solution. Therefore, minimizing the criterion amounts to ensuring the discrepancy calibration is accurate near $\zbar_k$, on average. Utilizing the Bayesian framework, our approach is also able to design batches of informative inputs simultaneously with $p > 1$.

\boldheading{Selection of $\aoed$}
The choice of the scalar hyperparameter $\aoed > 0$ gives control over the trade-off between exploration (larger $\aoed$) and exploitation (smaller $\aoed$). However, in many cases, it can be difficult to determine $\aoed$ in advance due to limited information on the optimization landscape. Nevertheless, we provide a heuristic approach to determine a suitable choice of the parameter. We adaptively select $\aoed$ in accordance to the distance between subsequent optimal solution updates, i.e., such that
\begin{equation}
  \mathbb{E}_{\z \sim \mu_{k}}\|\z - \zbar_k\|_{\M_{\z}}^2 = \|\zbar_{k} - \zbar_{k-p}\|_{\M_{\z}}^2.
  \label{eq:aoed_motivation}
\end{equation}
Simplifying the expectation in \cref{eq:aoed_motivation} and solving for $\aoed$, we define the hyperparameter\footnote{Note that $\zbar_0=\tilde{\z}$, $\z_1=\tilde{\z}$, and $\zbar_1$ is the predicted optimal solution using the data $\{\z_1,S(\z_1)-\tilde{S}(\z_1)\}$, so $\alpha_1$ can be initialized using~\cref{eq:aoed}. Manual specification of $\alpha_k$ is not typically required.} for $k \geq p$ as 
\begin{equation}
  \aoed = \frac{\|\zbar_{k} - \zbar_{k-p}\|_{\M_{\z}}^2}{\tr(\W_{\z}^{-1}\M_{\z})}.
  \label{eq:aoed}
\end{equation}
Observe that this choice of $\alpha_k$ is proportional to the distance between subsequent optimal solution updates. In other words, the criterion prioritizes uncertainty reduction of the discrepancy in a smaller neighborhood of $\zbar_k$ as the distance between optimal solution updates becomes smaller, which serves an indicator of some convergence. This approach incentivizes exploration when optimal solution updates are distant and exploitation when optimal solution updates are close. Although \cref{eq:aoed} serves as a general heuristic, in applications where there is additional information on the variability of the model discrepancy, one can select $\aoed$ in a more informative way; e.g., a larger $\aoed$ can be supplied to motivate exploration if the discrepancy does not vary much in $\mu_k$.
 
When minimizing the OED criterion, we also enforce that data remains within one Mahalanobis unit of $\zbar_k$. Specifically, we augment~\eqref{eq:OED_crit_optim} with the additional constraint:
\begin{equation}
  \| \z_{\ell} - \zbar_{k}\|^2_{\W_{\z}} \leq \aoed,\qquad \textrm{for }\ell = k+1, k+2, \dots, k+p.
\end{equation}
In other words, new data point(s) are acquired no further from $\zbar_k$ than $\| \zbar_k - \zbar_{k-1}\|$. This ensures that the resulting data remains in a reasonable range. With this constraint, observe that the optimal solution update tracing approach corresponds to taking $p=1$ and letting $\aoed \to 0$, and hence is a special case of our approach.  

We discuss the efficient computation of the OED criterion and its gradient in Appendix~\ref{ssec:app_fast_OED}. Specifically, we leverage the Kronecker product structure of the model discrepancy representation and prior covariance matrix to express posterior quantities in terms of Kronecker products in the state and optimization variable spaces. With appropriate offline computations, we show that the OED criterion and its gradient can be evaluated with a leading cost of $\mc{O}((k+p)n_u)$ floating point operations. 

\section{Numerical results}
\label{sec:results}
In this section, we apply the HDSA-MD framework to two different illustrative applications. First, we consider a control problem for a nonlinear diffusion-reaction PDE. This application, covered in \Cref{ssec:dr_results}, serves as a sandbox to analyze and gain insight into various aspects of our proposed framework. This is followed by a contaminant source identification problem involving a fluid flow model, detailed in \Cref{ssec:NS_results}. This example serves a numerical application and highlights the potential of HDSA-MD equipped with OED and continuation for complex nonlinear models.

\subsection{Diffusion-Reaction Problem}
\label{ssec:dr_results}
Consider the optimal control problem,
\begin{align}
\label{eq:diff_react_obj}
\min_{z} J(S(z),z) \coloneqq \frac{1}{2} \int_0^1 (S(z)(x) - T(x))^2 dx + \frac{\gamma}{2} \int_0^1 z(x)^2 dx
\end{align}
where $z:[0,1] \to \R$ is the control variable, and $T(x) = 20(x+.5)(1.3-x)$ is a target state. The high-fidelity model $S(z)$ is defined as the solution operator to the diffusion-reaction model,
\begin{equation}
\begin{alignedat}{2}
-\kappa \difft{u}{x} + R(u) = z &\qquad& \text{on } (0,1) \\
\kappa \diff{u}{x} = 0 &\qquad& \text{on } \{0,1\} 
\end{alignedat}
\label{eq:DR_PDE}
\end{equation}
where $\kappa>0$ is the diffusion coefficient and $R(u) = (1+.7 \sin(2 \pi x)) u^2$ is a spatially heterogeneous reaction function. We consider a low-fidelity model $\St(z)$ which simplifies physics by using a spatially homogeneous reaction function $\tilde{R}(u) = u^2$ in place of $R(u)$. Note that this model problem enables computation of both the low-fidelity solution $\tilde{z}$ and the high-fidelity solution $z^\star$, and thus provides a useful testbed for numerical experiments.
\subsubsection{Experiments on Continuation}
\label{sssec:exp_cont}
In this section, we analyze the effect of pseudo-time continuation for optimal solution updates (see \Cref{sec:continuation}). To isolate the role of continuation within the HDSA-MD framework, we fix (rather than estimate from data) the discrepancy parameter to a nominal value $\t = \t_{\text{nom}}$ obtained by linearizing the discrepancy $S(\z) - \St(\z)$ at the true high-fidelity optimal solution. We then perform continuation to obtain the optimal solution update $\zbar$ as described in \Cref{sec:continuation}. We compare the high-fidelity objective $\mathcal{J}(\zbar, \t_{\text{nom}})$, where a collection of $\zbar$'s are computed by varying the number of continuation steps. 
\Cref{fig:cont-steps} (left) displays the error in the objective function evaluated at the optimal solution updates. We measure error relative to the optimal value $\min_{\b \in \R^r} \mathcal{J}(\z_{\b}, \t_{\text{nom}})$. This eliminates any approximation error due to the quality of the projector.

In \Cref{fig:cont-steps} (left), observe that the relative error in the objective function consistently decreases with the number of continuation steps. It is notable that the error reduction due to additional continuation steps is significant initially, as the continuation approach begins to capture the nonlinearity in the optimization update trajectory. However, after a sufficient number of continuation steps, other sources of numerical error, such as those arising from linear solves, begin to dominate the optimization error, leading to minimal improvement with more refined continuation. For problems that we have considered, a moderate number of continuation steps (e.g., $3-4$) usually suffices.  

\begin{figure}[h]
  \centering
  \includegraphics[width=2.25in]{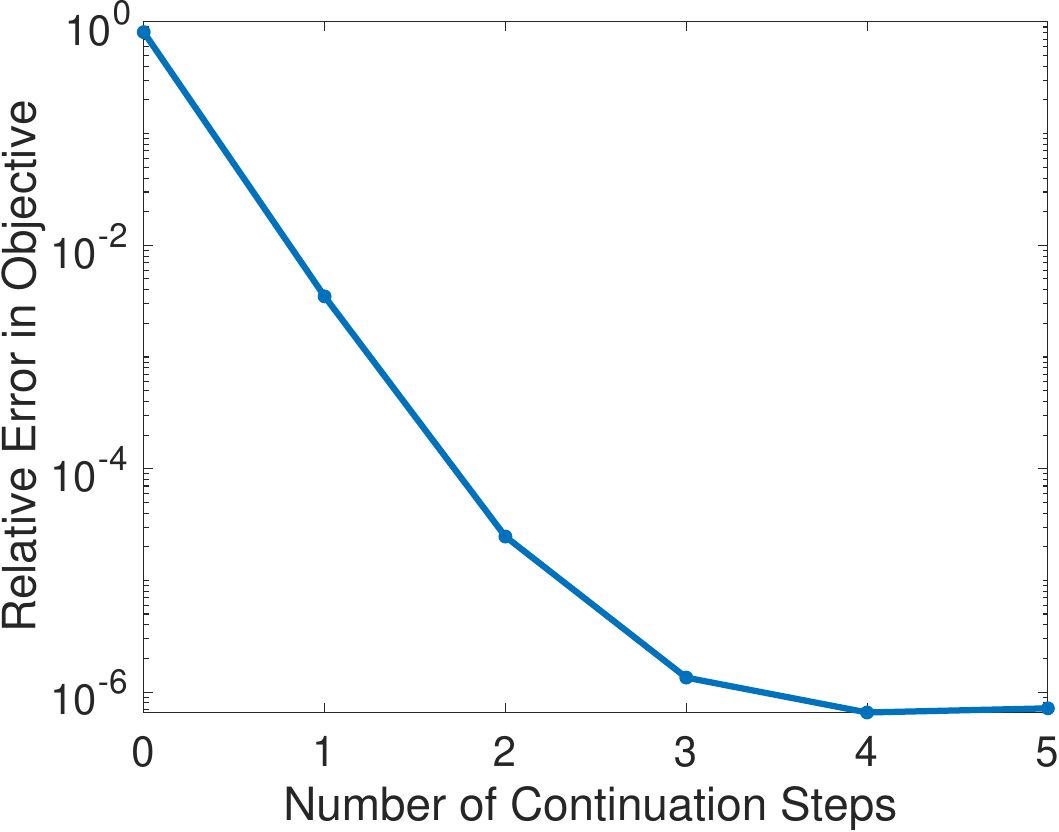}
  \qquad
  \includegraphics[width=2.5in]{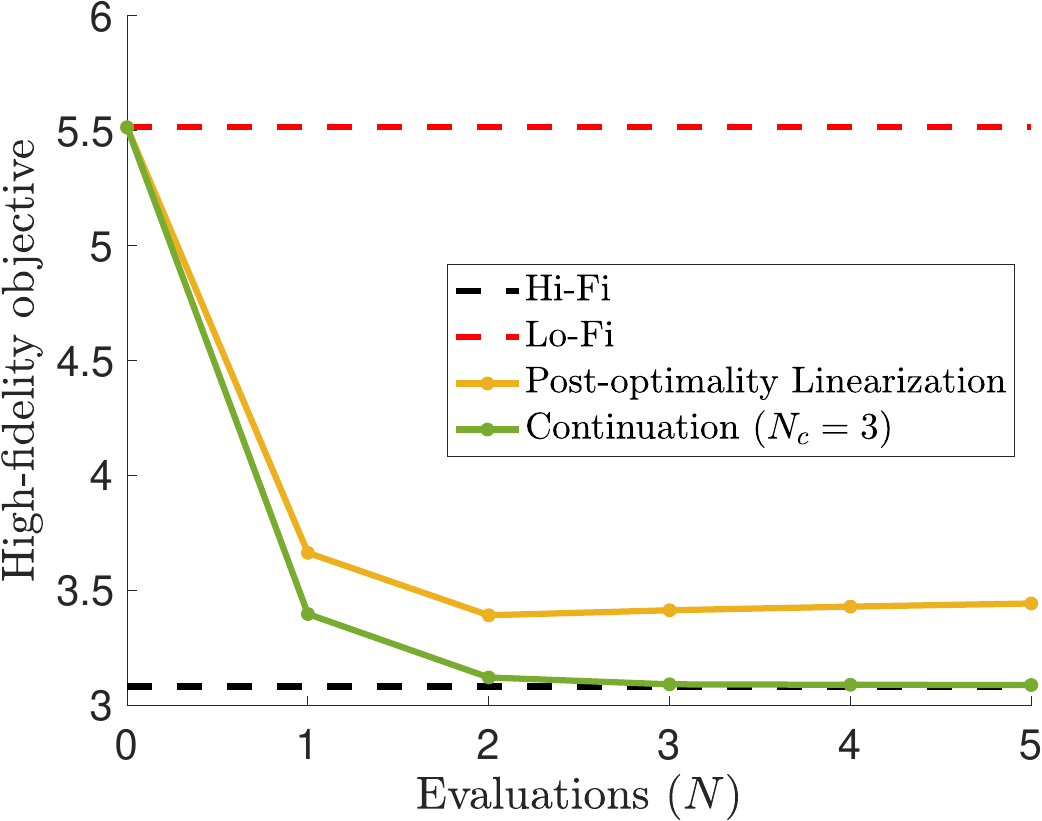}
  \caption{Optimization objective over continuation steps (left) and number of data points (right).}
  \label{fig:cont-steps}
\end{figure}

In addition to verifying the accuracy of continuation in isolation, we may observe the impact and role of continuation in conjunction with the rest of the framework. To this end, we perform model discrepancy calibration using a varying number of data points, obtain the optimal solution updates (with and without using continuation), and compare the resulting high-fidelity optimization objectives. Note that for this experiment, data is acquired upfront using the sequential OED approach equipped with continuation. This allows us to compare continuation with post-optimal linearization \cref{eq:lin_sol} when both approaches use the same data; we investigate optimal data acquisition in the context of this application further in \Cref{sssec:exp_oed}. In \Cref{fig:cont-steps} (right), we notice that with standard post-optimality linearization, there is conspicuous stagnation occurring due to the linearization error in the optimal solution updates, such that the optimal solution does not improve despite additional high-fidelity data. On the other hand, using continuation, even with a modest $N_c=3$ steps, offers a meaningful improvement in the quality of optimal solution. With just two or three high-fidelity evaluations, we obtain updated solutions that are $90\%$ closer to the best attainable optimization objective compared to the low-fidelity solution. Therefore, continuation allows for a substantially more accurate optimal solution update with a modest amount of additional computation. 

\subsubsection{Experiments on Optimal Data Acquisition}
\label{sssec:exp_oed}
In this section, we fix the continuation scheme with $N_c=3$ and turn to analyzing the effect of optimal data acquisition in uncertainty reduction and improvement in the quality of the optimal solution. Recall that optimal data acquisition is performed sequentially by minimizing the OED criterion $\Psi_k$ at each step, as outlined in \Cref{sec:OED}. To analyze the effectiveness of this approach, we compare this approach with random data points sampled from the distribution, $\mathcal{N}\left( \zbar_k, \aoed \ \P\W_{\z}^{-1}\P^\top \right)$. The choice of distribution mimics the search region for the optimal design, ensuring that data is sampled in a neighborhood of the optimal solution update along directions relevant to the optimization problem (in the subspace spanned by $\mat{V}$). 

In \Cref{fig:oed-rand-compare} (left), we compare the proposed optimal data acquisition method with random data acquisition. We sweep over a range of $N$'s (the number of high-fidelity evaluations used in the calibration) and compute the high-fidelity objective value at the updated optimal solution. Note that to ensure a fair comparison, the data points $\{\z_1, \dots, \z_k\}$ are obtained through OED when comparing the approaches at $k+1$. As observed in \Cref{fig:oed-rand-compare} (left), the proposed optimal data collection approach yields a better optimal solution update (i.e., a lower objective) compared to random data points on average (dashed line). This aligns with our theoretical observations in \Cref{ssec:oed_crit}. We do however point out that the quality of optimal solution updates cannot be solely used as a metric for the quality of data collection, as the former is also dependent on the profile of the underlying model discrepancy which is unknown a priori.

\begin{figure}[h]
  \centering
  \includegraphics[width=2.35in]{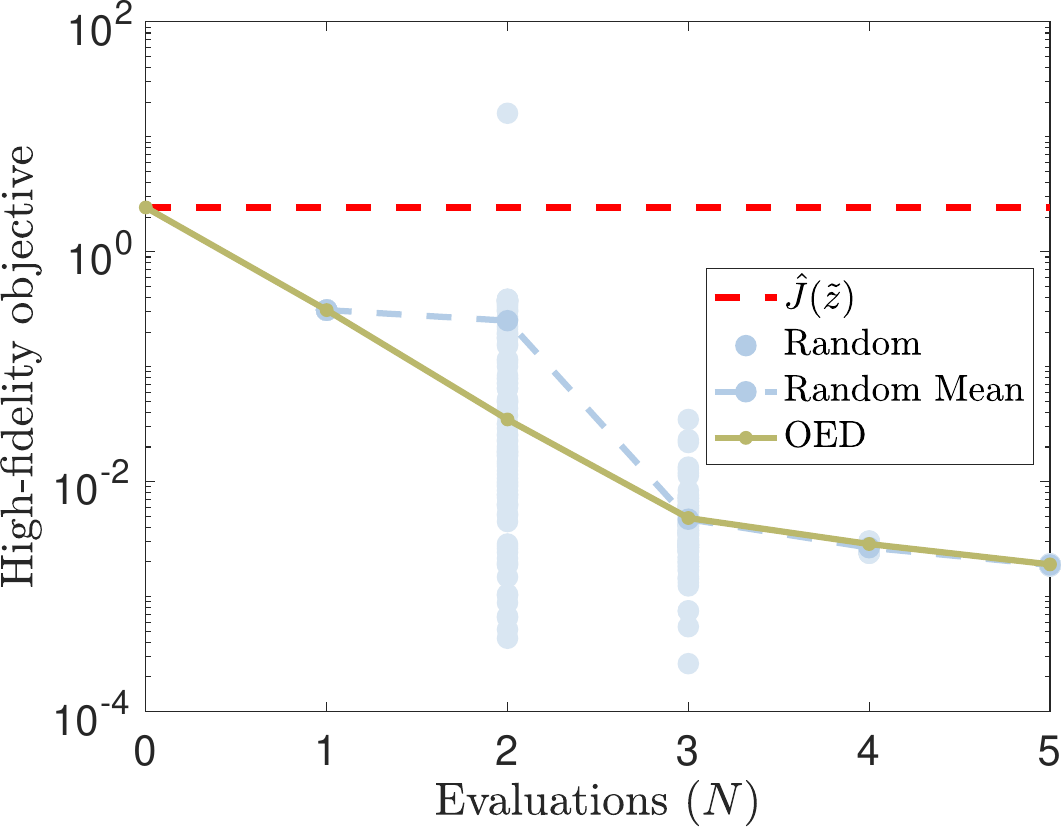}
  \qquad
  \includegraphics[width=2.45in]{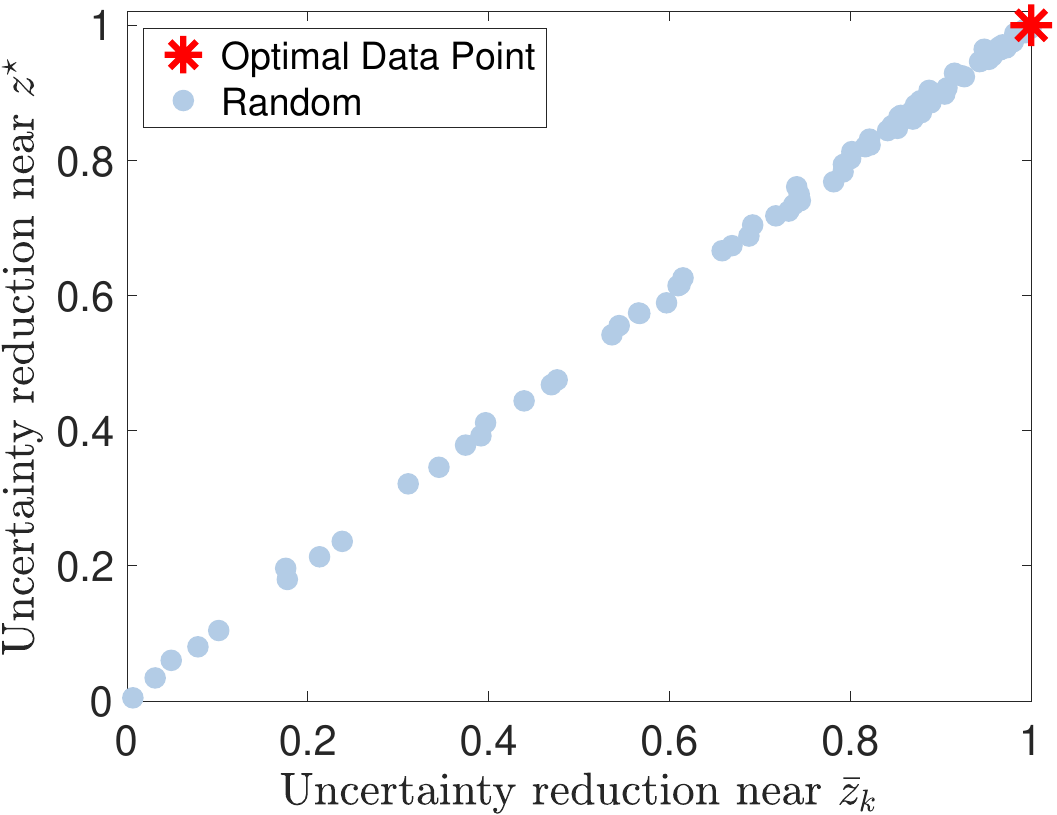}
  \caption{Comparison of objective (left) and OED criteria (right) with random data points.}
  \label{fig:oed-rand-compare}
\end{figure}

Therefore, to assess the effect of the optimal data acquisition in an isolated manner, we additionally analyze the uncertainty in the discrepancy near the high-fidelity optimal solution $\z^\star$. We quantify this uncertainty in an analogous way to $\Psi_k$ as the average variance in the discrepancy, $\Exp_{\z \sim \mu_k}[\tr(\vec{\Sigma}_\d (\z) \M_{\u})]$, over the neighborhood $\z \sim \mathcal{N}(\z^\star, \aoed \W_{\z}^{-1})$. In \Cref{fig:oed-rand-compare} (right), we compare the OED criterion and high-fidelity uncertainty measures for the third high-fidelity model evaluation $\z_3$. Specifically, we compare these metrics at the optimal data point obtained via OED as well as random data points. Since our computational methods ignore design-independent components of the uncertainty measures, we plot relative uncertainty reductions for interpretability. That is, the uncertainty measures are shifted and normalized by the uncertainty in the discrepancy if the $(k+1)$-th data point, $\z_{k+1}$, is placed again at the $k$-th data point, $\z_{k}$.

We observe in \Cref{fig:oed-rand-compare} (right) that the optimal data acquisition results in a noticeable uncertainty reduction in the model discrepancy near $\zbar_k$ compared to random data points. Note that the figure indicates a strong correlation between the uncertainty near $\zbar_k$ and $\z^\star$. This provides validation to the sequential OED approach which uses uncertainty near $\zbar_k$ as a surrogate to perform uncertainty reduction in a neighborhood of $\z^*$; although, the strength of this correlation depends on multiple factors including the distance between $\zbar_k$ and $\z^\star$ and the magnitude of $\alpha_k$. Nevertheless, the optimal data point obtained by the minimization of $\Psi_k$ achieves more uncertainty reduction near $\z^\star$ compared to nearly all random data points.

Lastly, we compare the performance of sequential model discrepancy calibration and data acquisition over the batch size $p$, in \Cref{fig:oed-compare-p}. We observe that data acquisition with larger batch sizes typically yields optimal solution updates with inferior quality. This is expected since sequential data acquisition exploits past discrepancy evaluations to better refine data acquisition in future steps; with larger batch sizes, this capacity becomes limited. Note that in this case, two strategic high-fidelity evaluations collected sequentially results in a better optimal solution update than six high-fidelity evaluations collected in parallel upfront. Nevertheless, given enough rounds of batches ($N \geq 2p$), a modest batch size can perform almost as well as a fully sequential approach. For example, this can be critical in applications where running high-fidelity model simulations consume hours or days in runtime, in which case parallel execution with a modest batch size can provide immense benefit. Therefore, the batch sequential OED approach ($p > 1$) can benefit from the added performance, resource utilization, and efficiency of computational routines in cases where parallel data collection is preferred.

\begin{figure}[ht]
  \centering
  \includegraphics[width=2.5in]{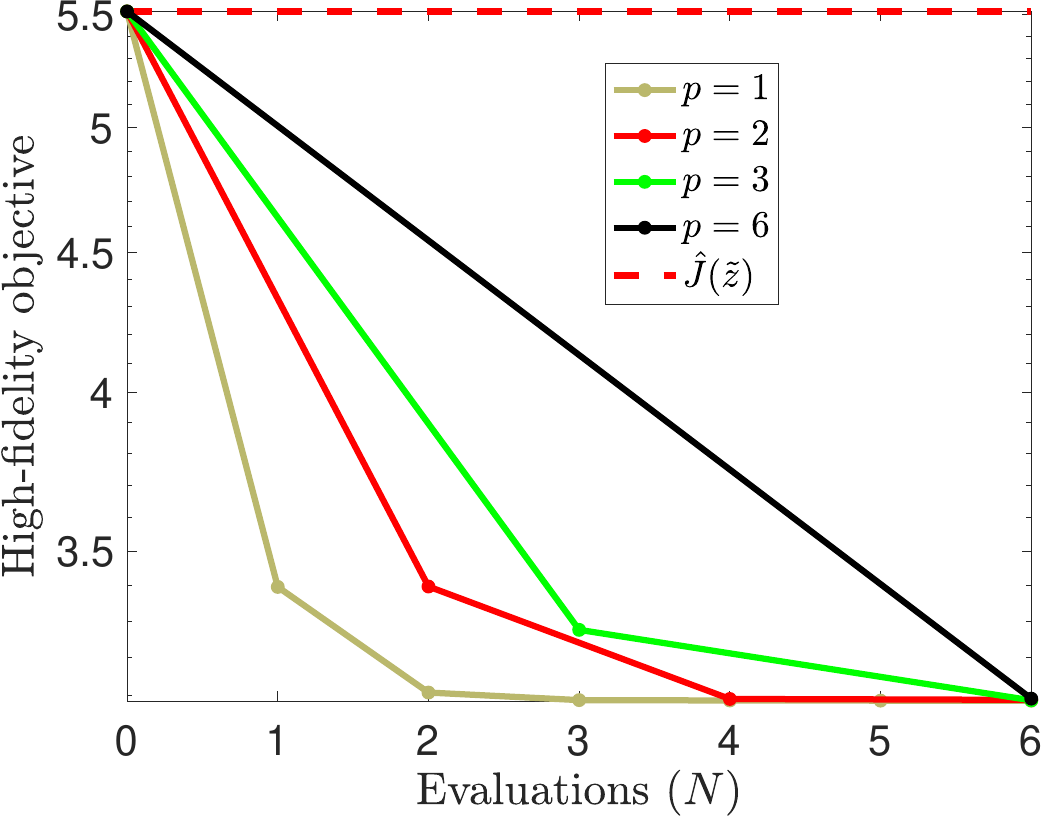}
  \caption{Performance of batch sequential OED with $p = 1, 2, 3$, and $6$.}
  \label{fig:oed-compare-p}
\end{figure}

\subsection{Application to Contaminant Source Identification}
\label{ssec:NS_results}
In this section, we apply our HDSA-MD approach to a fluid flow application involving contaminant source identification. Unlike the diffusion-reaction example, which served to provide insight into specific aspects of HDSA-MD, this example involves intricate couplings and serves to highlight the potential of our framework to more complex domains.
We consider a contaminant with concentration $c(x,t)$ that evolves in a one-dimensional fluid domain $\Omega=[0,1]$ over the time interval $0 \leq t \leq T$. The goal of the inverse problem is to recover the unknown initial concentration $z(x) := c(x,0)$ from the terminal concentration $\bar{c}(x) := c(x, T)$. We pose the reconstruction as the optimization problem,
\begin{equation}
\label{eq:ns_obj}
\min_{z} J(S(z),z) \coloneqq \frac{1}{2} \int_0^1 (S(z)(x) - \bar{c}(x))^2 dx + \frac{\gamma_{r}}{2} \int_0^1 z(x)^2 dx,
\end{equation}
where $S(z)$ is the high-fidelity mapping from the initial concentration to the terminal concentration and $\gamma_r = 10^{-6}$ is a regularization coefficient. The high-fidelity fluid flow model is governed by the 1D transient inviscid flow equations:
\begin{subequations}
\begin{alignat}{2}
  &\diff{\rho}{t} + u\diff{\rho}{x} + \rho\diff{u}{x} = 0 \label{eq:hifi-mass}\\[1mm]
  &\diff{u}{t} + u\diff{u}{x} + \frac{1}{\rho}\diff{p}{x} = 0 \label{eq:hifi-mom}\\[1mm]
  &\diff{e}{t} + u\diff{e}{x} + \frac{p}{\rho}\diff{u}{x} = \frac{k}{\rho c_v}\difft{e}{x} + \frac{\nu}{\rho} f(c) \label{eq:hifi-energy}\\[1mm]
  &p = \left( \frac{R}{c_v} \right)\rho e \label{eq:hifi-ideal}
\end{alignat}
\label{eq:ns_hifi}
\end{subequations}
Here, the velocity field $u$, the pressure field $p$, the density $\rho$, and the specific internal energy $e$ are scalar-valued functions of space and time. Moreover, $k, R, \nu$, and $c_v$ denote the thermal diffusivity constant, specific gas constant, the energy yield coefficient, and specific heat constant of the fluid, respectively. The contaminant concentration $c(x,t)$ satisfies the advection-diffusion-reaction PDE,
\begin{equation}
    \diff{c}{t} + u\diff{c}{x} + c\diff{u}{x} = \gamma_d\difft{c}{x} - f(c) \hspace{0.5in}
    \tag{\theequation e}
    \label{eq:contaminant}
\end{equation}
where $\gamma_d > 0$ is a diffusivity constant of the contaminant. The reaction term $f(c)$ represents a chemical reaction resulting in the conversion of the contaminant into energy in \cref{eq:hifi-energy}. The initial and boundary conditions used for the coupled flow--transport model are summarized in \Cref{tab:ns_params}.
\begin{table}[ht]
  \centering
  \renewcommand{\arraystretch}{1.5}
  \begin{tabular}{c|c|c}
    \textbf{Left Boundary (inflow)} & \textbf{Right Boundary (outflow)} & \textbf{Initial Condition} \\ \hline
    $u(0, t) = 1$ & --- & $u(x, 0) = 1 - x/2$ \\ \hline
    --- & $p(1, t) = 2$ & --- \\ \hline
    $\rho(0, t) = 1$ & --- & $\rho(x, 0) = 1$ \\ \hline
    $\diff{e}{x}(0, t) = 0$ & $\diff{e}{x}(1, t) = 0$ & $e(x, 0) = 1 + x$ \\ \hline 
      $\diff{c}{x}(0, t) = 0$ & $\diff{c}{x}(1, t) = 0$ & $c(x, 0) = z(x)$ 
        \end{tabular}
  \caption{Initial and boundary conditions for fluid flow $(0 \leq t \leq T = 0.1)$.}
  \label{tab:ns_params}
\end{table}

We utilize (nondimensionalized) coefficients $\gamma_d = 0.05$, $c_v = 1$, $k = 30$, $R = 1$, $\nu = 10^5$, and reaction function $f(c) = c$ for this test application. Notice that \cref{eq:ns_hifi} fully couples energy, velocity, and the contaminant concentration: the contaminant concentration affects the internal energy of the fluid, which changes its velocity, which in turn affects the evolution of the contaminant. Hence, the mapping $S$ is highly nonlinear and solving the optimization problem \cref{eq:ns_obj} is nontrivial.

We observe, however, that the nonlinearity of the inverse problem hinges on the effects of the contaminant on the velocity field. Alternatively, one can approximate the contaminant concentration by fixing the spatiotemporal velocity profile at a nominal setting and only solving~\eqref{eq:contaminant} to predict the contaminant transport. This approximation, which we denote as $\tilde{S}$, is advantageous in the sense that it significantly lessens the complexity of the governing model and facilitates optimization: the low-fidelity solution operator depends linearly on $z$ and the optimization objective is a quadratic function of $z$. However, such an approximation ignores the effect of concentration on the velocity field and thus induces model error. 

The HDSA-MD framework, incorporated with post-optimal continuation and optimal data acquisition, enables an update of the initial condition estimate with only a few ($\le 5$) queries of the high-fidelity forward model. We apply this framework to the contaminant source identification problem and visualize the results in \Cref{fig:ns_results}. 
 In particular, \Cref{fig:ns_results} (left) displays the high-fidelity optimization objective, evaluated at the updated optimal solution, as we vary over the number of high-fidelity evaluations $N$ (using $p=1$). We observe that even under this complex setting, HDSA-MD manages to reduce the high-fidelity optimization objective by two orders of magnitude compared to the low-fidelity optimizer. Moreover, there is consistent improvement with additional data points. Although it is not easily seen in \Cref{fig:ns_results} due to the axis scaling, the objective function continues to decrease with each additional data point, with an average of 42\% improvement in optimization objective with each step. \Cref{fig:ns_results} (right) showcases the high-fidelity solution, low-fidelity solution, and learned optimal solution $\zbar_5$; a significant improvement over the low-fidelity solution is observed.

\begin{figure}[H]
  \centering
  \includegraphics[width=0.45\textwidth]{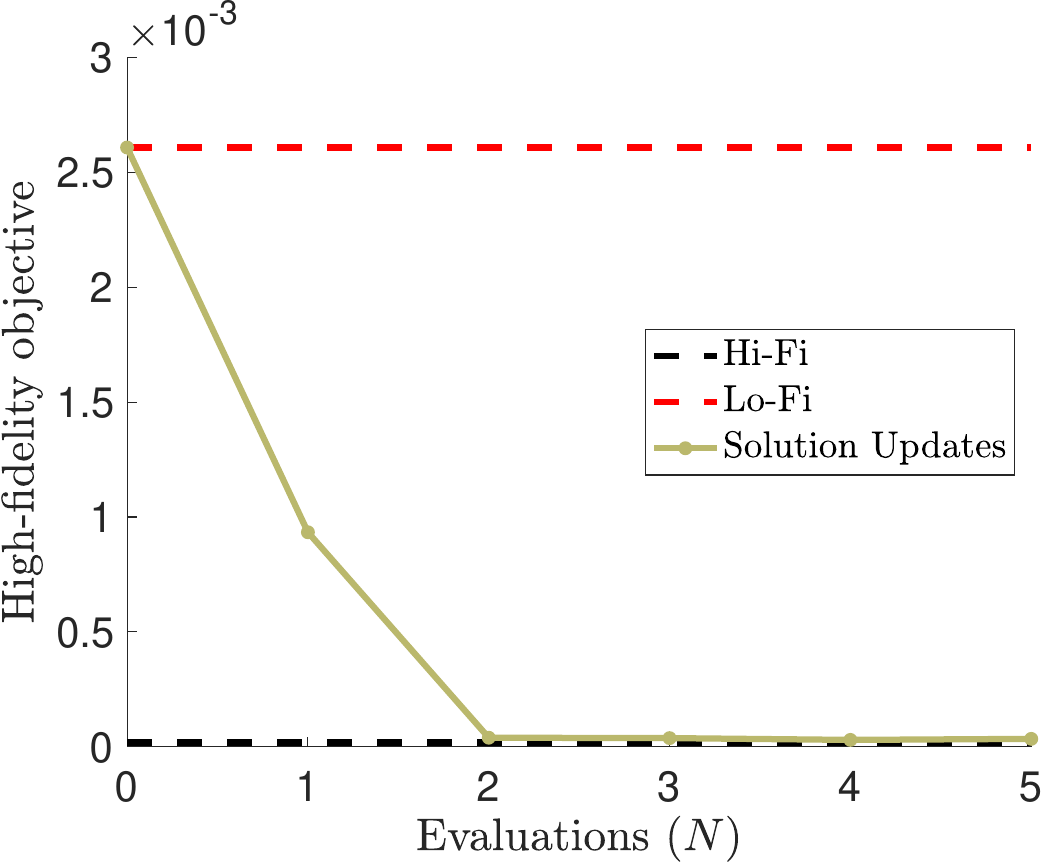}
  \qquad
  \includegraphics[width=0.45\textwidth]{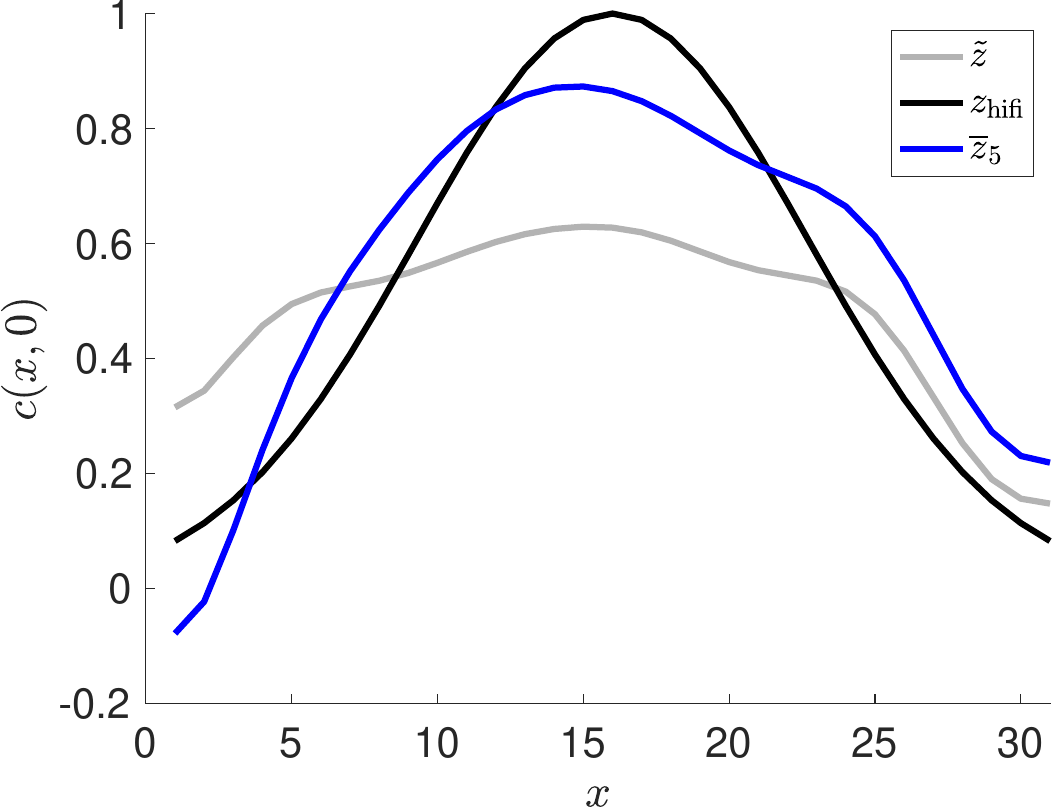}
  \caption{High-fidelity objective over the number of high-fidelity model evaluations (left), and the optimization solutions (right).}
  \label{fig:ns_results}
\end{figure}

\section{Conclusion}
\label{sec:conclusion}
The HDSA-MD framework is a promising approach to learn solutions of large-scale optimization problems, enabling improved optimization solutions through limited high-fidelity model evaluations and Bayesian uncertainty quantification. In this work, we have extended this framework by addressing key limitations associated with efficient data acquisition and linearization errors in optimal solution updates. Guiding the acquisition of high-fidelity model evaluations is especially challenging for model discrepancy calibration due to limited a priori information on the high-fidelity model. Our proposed sequential approach for optimal data acquisition enables targeted uncertainty reduction in discrepancy calibration and improves the expected quality of optimization solutions. Moreover, the utilization of pseudo-time continuation is shown to offer highly-accurate post-optimal propagation of model discrepancy, avoiding the need to resolve PDE-constrained optimization problems when learning optimal solution updates. Our numerical results further demonstrate the efficacy of our approach in a diffusion-reaction problem and coupled contaminant-fluid flow evolution model. The proposed approaches for model calibration is suitable for many other areas in science and engineering where model discrepancy and nonlinearity is prevalent, such as additive manufacturing and fusion energy.

There are several promising directions for further investigation. Our proposed approach depends on a heuristic to specify $\alpha_k$, the exploration/exploitation trade-off hyperparameter. Further work should explore richer theoretical insights to guide adaptation $\alpha_k$. Additionally, our proposed  approach restricts the experimental design and optimal solution update to a subspace defined by the low-fidelity objective function Hessian at the low-fidelity solution. Other approaches based on sample average of Hessians or subspace adaptation in the sequential algorithm may be advantageous. A key limitation with HDSA-MD is that an affine model could be inadequate for sufficient discrepancy calibration in certain applications. One possible approach in this direction involves utilizing heteroscedastic noise models to ingrain the discrepancy calibration with a sense of locality. This would incentivize the model calibration to place heavier emphasis on discrepancy data collected near the optimal solution updates. Not only would this approach better represent nonlinear model discrepancy, it would also enable more robust optimal data acquisition methods, e.g., reducing the need for explicit constraints limiting the design space. Similarly, piece-wise affine models, radial basis functions, and sequential updates to the prior that enforce locality can be considered in a similar vein, although such approaches may pose additional theoretical and computational complexities that need to be addressed. Advancements in efficient posterior uncertainty quantification of optimal solution updates using continuation could also better facilitate \textit{a posteriori} assessment of solution quality. 

\section*{Acknowledgements}
This work was supported by the Laboratory Directed Research and Development program at Sandia National Laboratories, a multimission laboratory managed and operated by National Technology and Engineering Solutions of Sandia LLC, a wholly owned subsidiary of Honeywell International Inc. for the U.S. Department of Energy's National Nuclear Security Administration under contract DE-NA0003525. SAND2026-19369O.


\bibliographystyle{siam}
\bibliography{references}


\appendix\section{Scalable computational methods for OED}
This section details the theoretical results and methods necessary for scalable implementation of optimal experimental design. In \Cref{ssec:app_post_fact}, we derive the necessary theoretical constructs that provide scalable representations of the posterior covariance matrix of $\t$. This is followed by \Cref{ssec:app_fast_OED}, in which we briefly describe the efficient optimization of OED criterion presented in \Cref{ssec:oed_crit}.
\subsection{Expressions for posterior computations}
\label{ssec:app_post_fact}
Recall that the posterior covariance matrix \cref{eq:post_mean_product} is an $n_\theta \times n_\theta$ operator for which explicit computation is typically infeasible. Nevertheless, owing to the Kronecker structure representation of the prior covariance matrix, noise covariance matrix, and forward operator, we may efficiently express and compute with the posterior covariance, $\mat{\Sigma}_\t$. This involves three key steps: (i) factorizing the prior covariance matrix, (ii) factorizing $\A$, and (iii) factorizing and inverting the posterior covariance matrix. The derivations presented in this section are similar to those found in \cite[Appendix B]{HDSA_MD_3}; however, a key distinction is that we utilize a different formulation of the prior precision matrix $\W_\t$. Scalable applications of the posterior covariance matrix are necessary for efficient posterior sampling and uncertainty quantification. In particular, the fast computation of the OED criterion in \Cref{ssec:oed_crit} fundamentally relies on the derivations presented in this section.

\subsubsection*{Factorizing $\W_\u$ and $\W_\t$}
Let $(\vec{x}_j,\lambda_j)$ denote the generalized eigenvalues and eigenvectors of $\W_\u$ in the $\M_\u$ weighted inner product.  Let $\vec{X} = \begin{pmatrix} \vec{x}_1 & \vec{x}_2 & \cdots & \vec{x}_{n_u} \end{pmatrix} \hspace{-1mm}$ and $\vec{\Lambda}$ be the diagonal matrix with entries $\lambda_1,\lambda_2,\dots,\lambda_{n_u}$, so that the generalized eigenvalue decomposition is $\W_\u = \M_\u \vec{X} \vec{\Lambda} \vec{X}^\top \M_\u$. Then we may decompose $\W_\t = \L \L^\top$ where
\begin{align}
\label{eq:L}
\L =
\left(
\begin{array}{cc}
\M_\u \vec{X} \vec{\Lambda}^\frac{1}{2} & \vec{0} \\
\M_\u \vec{X} \vec{\Lambda}^\frac{1}{2}  \otimes \M_\z \ztilde & \M_\u \vec{X} \vec{\Lambda}^\frac{1}{2}  \otimes \W_\z^{\frac{1}{2}} 
\end{array}
\right) .
\end{align}
Furthermore, thanks to its block lower triangular structure, the inverse of $\L$ is
\begin{align}
\label{eq:Linv}
\L^{-1} =
\left(
\begin{array}{cc}
\vec{\Lambda}^{-\frac{1}{2}} \vec{X}^\top & \vec{0} \\
\vec{\Lambda}^{-\frac{1}{2}} \vec{X}^\top   \otimes \left( -\W_\z^{-\frac{1}{2}} \M_\z \ztilde \right) & \vec{\Lambda}^{-\frac{1}{2}} \vec{X}^\top  \otimes  \W_\z^{-\frac{1}{2}}
\end{array}
\right) 
\end{align}
and the inverse of $\W_\t$ may be computed as $\W_\t^{-1} = \L^{-T} \L^{-1}$.
\subsubsection*{Factorizing $\A$}
Let $\A = \vec{\Xi} \vec{\Phi} \vec{\Psi}^\top \W_\t$ be the generalized singular value decomposition (GSVD) of $\A$ with a $\M_\u$-weighted inner product on its row space and a $\W_\t$-weighted inner product on its column space. Here, $\vec{\Phi}$ is the diagonal matrix of singular values, and $\vec{\Xi}$ and $\vec{\Psi}$ are matrices containing the singular vectors which satisfy $\vec{\Xi}^\top \M_\u \vec{\Xi} = \I_{n_u}$ and $\vec{\Psi}^\top \W_\t \vec{\Psi} = \I_{n_\theta}$. To determine the singular vectors, note that 
\begin{align*}
( \I_N \otimes \M_\u) \A \W_\t^{-1} \A^\top ( \I_N \otimes \M_\u)= ( \I_N \otimes \M_\u) \vec{\Xi} \vec{\Phi}^2 \vec{\Xi}^\top( \I_N \otimes \M_\u).
\end{align*}
 Hence, $\vec{\Xi}$ and $\vec{\Phi}$ may be determined from the generalized eigenvalue decomposition of
$( \I_N \otimes \M_\u) \A \W_\t^{-1} \A^\top ( \I_N \otimes \M_\u)$
 in the $(\I_N \otimes \M_\u)$ weighted inner product, or equivalently the eigenvalue decomposition of $\A \W_\t^{-1} \A^\top ( \I_N \otimes \M_\u)$. Using the expressions~\cref{eq:Aell} for $\A$ and~\cref{eq:Linv} for $\W_\t^{-1}=\L^{-T} \L^{-1}$, we can write
\begin{eqnarray*}
 \A \W_\t^{-1} \A^\top ( \I_N \otimes \M_\u) =\vec{G} \otimes  \W_\u^{-1} \M_\u
\end{eqnarray*}
where
\begin{align*}
& \G= \e \e^\top + (\vec{Z} - \ztilde \e^\top)^\top \M_\z\W_\z^{-1}\M_\z (\vec{Z}-\ztilde \e^\top) \in \R^{N \times N},
\end{align*}
$\vec{Z} = \begin{pmatrix} \z_1 & \z_2 & \dots & \z_N \end{pmatrix}$ is the matrix of optimization variable data, and $\e \in \R^{N}$ is the vector of ones. Hence, the left singular vectors $\vec{\Xi}$ correspond to the eigenvectors of $\vec{G} \otimes \W_\u^{-1} \M_\u$ and the squared singular values $\vec{\Phi}^2$ correspond to the eigenvalues.

Let $(\vec{g}_i,\mu_i)$ denote the eigenpairs of $\vec{G}$, and note that $(\vec{x}_j,\frac{1}{\lambda_j})$ are the eigenpairs of $\W_\u^{-1} \M_\u$. Recalling properties of the eigenvalue decomposition of a Kronecker product, we observe that the squared generalized singular values of $\vec{A}$ are given by $\frac{\mu_i}{\lambda_j}$ and are associated with the left singular vectors $\vec{\xi}_{i,j}=\vec{g}_i \otimes \vec{x}_j$. Rewriting the GSVD to solve for $\vec{\Psi}$, the right singular vector associated with $\vec{\xi}_{i,j}$ is given by
 \begin{eqnarray}
 \label{eq:psi_ij}
\vec{\psi}_{i,j}= \frac{1}{\sqrt{\mu_i \lambda_j}} 
\left(
 \begin{array}{cc}
s_i  \vec{x}_j \\
 \vec{x}_j \otimes \W_\z^{-1} \M_\z \vec{y}_i
 \end{array}
 \right), \quad i=1,2,\dots,N, \ \ j=1,2,\dots,n_u,
 \end{eqnarray}
where
\begin{eqnarray}
\label{eq:w_vecs}
\vec{y}_i = \vec{Z} \vec{g}_i - (\e^\top \vec{g}_i) \ztilde \qquad \text{and} \qquad s_i = (\e^\top \vec{g}_i) - \vec{y}_i^\top \M_\z \W_\z^{-1} \M_\z \ztilde .
\end{eqnarray}
Given these decompositions, we express $\vec{\Sigma}_\t^{-1}$ as
\begin{eqnarray}
\label{eq:sigma_inv_factor}
\vec{\Sigma}_\t^{-1} =  \frac{1}{\alpha_{\vec{d}}} \L \C \L^\top, \quad\text{where}\quad
\C =  \alpha_{\vec{d}} \vec{I} +  \L^\top \vec{\Psi} \vec{\Phi}^2 \vec{\Psi}^\top \L .
\end{eqnarray}

\subsubsection*{Inverting $\vec{\Sigma}_\t^{-1}$}
This factorization of $\vec{\Sigma}_\t^{-1}$ implies that $\vec{\Sigma}_\t= \alpha_{\vec{d}} \L^{-T} \C^{-1} \L^{-1}$ and facilitates easy manipulation because $\C$ is a composition of diagonal and orthogonal matrices. Applying the Sherman-Morrison-Woodbury formula to $\C$ in \eqref{eq:sigma_inv_factor}, we have
\begin{eqnarray*}
\C^{-1} = \frac{1}{\alpha_{\vec{d}}} \left( \vec{I} - \L^\top \vec{\Psi} \vec{D}\vec{\Psi}^\top \L \right)
\end{eqnarray*}
where $\vec{D}\in \R^{n_uN \times n_uN}$ is a diagonal matrix whose entries are given by $\frac{\mu_i}{\mu_i + \alpha_{\vec{d}} \lambda_j}.$
Algebraic simplifications yields
\begin{eqnarray}
\label{eq:Sigma}
\vec{\Sigma}_\t = \W_\t^{-1} - \vec{\Psi} \vec{D}\vec{\Psi}^\top .
\end{eqnarray}

\subsection{Fast computation of the OED criterion}
\label{ssec:app_fast_OED}
In this section, we derive the necessary expressions for computation of the criterion. Suppose that all data points are chosen in the projected subspace according to \cref{eq:z_ell_proj}. Denote the design parameters as $\vec{\eta}=( \vec{\eta}^{k+1},\vec{\eta}^{k+2},\dots,\vec{\eta}^{k+p}) \in \R^{pr}$ where $\vec{\eta}^\ell=(\eta_1^\ell,\eta_{2}^\ell,\dots,\eta_r^\ell) \in \R^r$. Observe that this notation enables us to express, 
\[\mat{Z} = \zt \vec{e}^\top + \mat{V} \mat{M}(\vec{\eta}), \quad \text{where} \quad\mat{M}(\vec{\eta}) = \begin{bmatrix}
  \vec{0} & \vec{\eta}^2 & \cdots & \vec{\eta}^{k+p}
\end{bmatrix}.\]
This formulation allows us to minimize the OED criterion over the lower-dimensional space of $\vec{\eta}$. To this end, recall that the posterior distribution of $\t$ is given by $\t \sim \mathcal{N}(\bar{\t}, \vec{\Sigma}_\t)$, where $\vec{\Sigma}_\t = \W_\t^{-1} - \vec{\Psi}\vec{D}\vec{\Psi}^\top$. We express
\begin{eqnarray}
\delta(\z,\t) = \A(\z) \t,\quad\text{where}\quad
  \A(\z) = \left( \begin{array}{cc}\I_{n_u} & \I_{n_u} \otimes \z^\top \M_\z \end{array} \right).
\end{eqnarray}
Therefore, for a given $\z$, the covariance matrix of $\delta(\z,\t)$ with respect to the posterior distribution of $\t$ can be written as $\vec{\Sigma}_\d (\z) = \A(\z)\vec{\Sigma}_\t\A(\z)^\top$. By \cref{eq:Sigma}, we can express
\begin{align*}
  \tr(\vec{\Sigma}_\d (\z) \M_\u) 
  &= \tr(\A(\z)\W_\t^{-1}\A(\z)^\top\M_\u) - \tr(\A(\z)\vec{\Psi}\vec{D}\vec{\Psi}^\top\A(\z)^\top\M_\u)
\end{align*}
Taking expectation over $\z \sim \mu_k$, we obtain
\begin{align*}
  \Psi_k(\vec{\eta})
    &=  \mathbb{E}_{\mu_k}\! \left\{ \tr(\A(\z)\W_\t^{-1}\A(\z)^\top\M_\u) \right\} -  \mathbb{E}_{\mu_k}\! \left\{ \tr(\A(\z)\vec{\Psi}\vec{D}\vec{\Psi}^\top\A(\z)^\top\M_\u) \right\}.
\end{align*}
Since $\W_\t^{-1}$ is design-invariant, minimizing $\Psi_k(\vec{\eta})$ is equivalent to maximizing,
\[\widetilde{\Psi}_k(\vec{\eta}) := \mathbb{E}_{\mu_k} \left\{ \tr(\A(\z)\vec{\Psi}\vec{D}\vec{\Psi}^\top\A(\z)^\top \M_\u) \right\}\]
Substituting~\eqref{eq:psi_ij} into the decomposition, we can expand 
\begin{align*}
  \tr(\A(\z)\vec{\Psi}\vec{D}\vec{\Psi}^\top\A(\z)^\top\M_\u) = \sum_{i=1}^{N}\sum_{j=1}^{n_u} \frac{\x_j^\top \M_\u \x_j}{\lambda_j\left(\mu_i+\alpha_{\vec{d}}\lambda_j\right)}\,
 \Bigl[s_i+ \z^\top\M_\z\W_{\z}^{-1}\M_\z\y_i\Bigr]^2.
\end{align*}
Taking expectation with respect to the measure $\mu_{k}$, we obtain,
\begin{align} \label{eqn:oed_cri}
\widetilde{\Psi}_k(\vec{\eta}) & \coloneq 
\sum_{i=1}^{N}\sum_{j=1}^{n_u} p_i \frac{\x_j^\top \M_\u \x_j}{\lambda_j\left(\mu_i+\alpha_{\vec{d}}\lambda_j\right)} \nonumber \\
& = \sum_{i=1}^{N} p_i \text{Tr}(\M_\u \W_\u^{-1} \M_\u \left( \alpha_{\vec{d}} \W_\u + \mu_i \M_\u \right)^{-1})
,
\end{align}
where 
\[p_i := \left(s_i + \zbar_k^\top\M_\z\W_{\z}^{-1}\M_\z\y_i\right)^2 + \aoed \ \y_i^\top(\M_{\z}\W_{\z}^{-1}\M_{\z})\W_{\z}^{-1}(\M_{\z}\W_{\z}^{-1}\M_{\z})\y_i.\]
Note that the criterion depends on the design-dependent quantities $\y_i$, $s_i$, and $\mu_i$ for $1 \leq i \leq N$. On the other hand, $\x_j^\top \M_\u \x_j$ and $\lambda_j$ for $1 \leq j \leq n_u$ are design-independent and can be evaluated in an offline manner. Similarly, the computation of $p_i$ can be accelerated by precomputing the application of $(\M_\z\W_\z^{-1}\M_\z)$ to $\mat{V}$, owing to the structure of $\{\z_{\ell}\}_{\ell=1}^N$ as shown in \cref{eq:z_ell_proj}.

In addition to the OED criterion, we also require fast evaluations of its gradient to enable efficient optimization for the design. Note that the criterion specifically depends on the eigenpairs of $\G$, $\{(\mu_i,\vec{g}_i)\}_{i=1}^N$, as $\mu_i$ appears in~\Cref{eqn:oed_cri} and $\y_i$ and $s_i$ depend on $\vec{g}_i$. Therefore, evaluating $\nabla_{\vec{\eta}} \widetilde{\Psi}_k(\vec{\eta})$ requires computation of the derivatives of the eigenpairs of $\G$, $(\mu_i,\vec{g}_i)$, with respect to $\vec{\eta}$. For the following expressions, we assume that the eigenvalues are simple (i.e., distinct) and use the convention of normalized eigenvectors $\{\vec{g}_i\}_{i=1}^n$ with the first nonzero element positive in each vector (to prevent sign flipping). Following~\cite{Abou-Moustafa}, we have 
\begin{align*}
\frac{\partial \mu_i}{\partial \eta_k^\ell}= \g_i^\top \left( \frac{\partial \G}{{\partial \eta_k^\ell}} \right) \g_i \qquad \text{and} \qquad \frac{\partial \g_i}{\partial \eta_k^\ell}= (\mu_i \I - \G)^\dagger \left(\frac{\partial \G}{\partial \eta_k^\ell}\right) \g_i .
\end{align*}
Given the expression $ \G = \e \e^\top + \vec{M}(\vec{\eta})^\top (\V^\top\M_\z\W_\z^{-1}\M_\z \V) \vec{M}(\vec{\eta})$, we express
\begin{align*}
\frac{\partial \G}{\partial \eta_k^\ell} =  \frac{\partial \vec{M}}{\partial \eta_k^\ell}(\vec{\eta})^\top (\V^\top\M_\z\W_\z^{-1}\M_\z \V) \vec{M}(\vec{\eta}) +  \vec{M}(\vec{\eta})^\top (\V^\top\M_\z\W_\z^{-1}\M_\z \V)\frac{\partial \vec{M}}{\partial \eta_k^\ell}(\vec{\eta})
\end{align*}
where $\frac{\partial \vec{M}}{\partial \eta_k^\ell}(\vec{\eta})$ is the matrix with zero in every entry except $(k,\ell)$ which has the entry one. Therefore, we may compute the gradient of $\mu_i$ and Jacobian of $\g_i$, for $i=1,2,\dots,N$, using matrix products in dimension $N$.

\end{document}